\documentclass[10pt]{amsart}

\usepackage{amsfonts,amsmath,amssymb,amsthm}
\usepackage{tikz}
\usepackage{enumerate}
\usepackage[colorlinks,linkcolor=blue!50!black,citecolor=green!50!black]{hyperref} 

\newtheorem{thm}{Theorem}[section]
\newtheorem{prop}[thm]{Proposition}
\newtheorem{lemma}[thm]{Lemma}

\theoremstyle{definition}

\usetikzlibrary{positioning,arrows,backgrounds,calc,through,fit} %
\tikzstyle{point}=[shape=circle,fill,inner sep=0.5mm]

\newcommand{\enne}{{\mathbb{N}}}
\newcommand{\ze}{{\mathbb{Z}}}
\newcommand{\qu}{{\mathbb{Q}}}
\newcommand{\erre}{{\mathbb{R}}}

\newcommand{\pro}{{\mathbb{P}}}
\newcommand{\cappa}{{\mathbb{K}}}
\newcommand{\effe}{{\mathbb{F}}}

\newcommand{\trop}{{\mathbb{T}}}
\newcommand{\mcors}{{\mathfrak{m}}}
\newcommand{\ocors}{{\mathcal{O}}}
\newcommand{\freccia}{{\longrightarrow}}
\newcommand{\hilb}[2]{\ensuremath{\operatorname{Hilb}_{#2}(#1)}}

\newcommand{\p}[1]{\ensuremath{{\mathbb P}^{#1}}}

\newcommand{\G}[2]{\ensuremath{{\mathbb G}(#1,#2)}}

\DeclareMathOperator{\adeg}{arith-deg}
\DeclareMathOperator{\ini}{in}
\DeclareMathOperator{\linspan}{Span}
\DeclareMathOperator{\mult}{mult}
\DeclareMathOperator{\reg}{reg}
\DeclareMathOperator{\sat}{sat}
\DeclareMathOperator{\cp}{cp}
\DeclareMathOperator{\tropicalization}{Trop}

\newcommand{\nuovo}[1]{{{\bfseries \upshape #1}}}
\newenvironment{thm*}[1]{\vspace{0.2cm}{\bf Theorem #1:} \itshape }{\vspace{0.2cm}}

\begin{document}

\title{On the Tropicalization of the Hilbert~Scheme}

\author[D.~Alessandrini]{Daniele Alessandrini}
\address{Daniele Alessandrini\\Institut de Recherche Math\'ematique Avanc\'ee\\CNRS et UDS\\7 rue Ren\'e Descartes\\67084 Strasbourg Cedex\\France}
\email{daniele.alessandrini@gmail.com}

\author[M.~Nesci]{Michele Nesci}
\address{Michele Nesci\\Universit\'e de G\`eneve, Section de Math\'ematiques\\2-4 rue du Lièvre, Case postale 64\\ 1211 Genève 4\\ Suisse}
\email{Michele.Nesci@unige.ch}

\date{\today}

\begin{abstract}
In this article we study the tropicalization of the Hilbert scheme and its suitability as a parameter space for tropical varieties. We prove that the points of the tropicalization of the Hilbert scheme have a tropical variety naturally associated to them. To prove this, we find a bound on the degree of the elements of a tropical basis of an ideal in terms of its Hilbert polynomial.

As corollary, we prove that the set of tropical varieties defined over an algebraically closed valued field only depends on the characteristic pair of the field and the image group of the valuation.

In conclusion, we examine some simple examples that suggest that the
definition of tropical variety should include more structure than what is currently considered.
\end{abstract}

\maketitle

\section{Introduction}
In \cite{SS} Speyer and Sturmfels studied the tropicalization of the Grassmannian and found that it is a parameter space for tropical linear subspaces, just like the ordinary Grassmannian in algebraic geometry. This result inspired our study of the tropicalization of the Hilbert scheme for two reasons: first, Grassmannians are an example of Hilbert schemes, and second, the standard construction of the Hilbert scheme realizes it as a subscheme of a Grassmannian.

The Hilbert scheme is a parameter space for embedded projective varieties. For this reason, it is natural to ask wether its tropicalization can be considered as a parameter space for tropical varieties. Our main result is the following:

\begin{thm*}{\ref{teo:teorema_principale}}
Let $\cappa$ be an algebraically closed field with surjective real valued valuation. There is a commutative diagram
\[
\begin{tikzpicture}[node distance=6cm]
 \node (G) {\hilb{p}{n}};
 \node[right of=G] (pl) {$\left\{ V \subset \cappa\pro^n \ |\ V \mbox{ has Hilbert polynomial } p \right\}$};
  \draw[<->] (G)-- node[auto] {$b$} (pl);
 \node[below=1cm of G] (TG) {$\tropicalization(\hilb{p}{n})$};
 \node[right of=TG] (Tpl) {$\left\{ \tropicalization(V) \subset \trop\pro^n \ |\ V \mbox{ has Hilbert polynomial } p \right\}$};
  \draw[->]  (G)   -- node[auto]{$\tau$} (TG);
  \draw[->>] (TG)  -- node[auto]{$s$}                   (Tpl);
  \draw[->]  (pl)  -- node[auto]{$\tropicalization$} (Tpl);
\end{tikzpicture}
\]
Where $b$ is the classical correspondence between points of the Hilbert scheme and subschemes of $\cappa\pro^n$. The map $s$ is surjective.
\end{thm*}

In order to prove this theorem, we need to prove a result about tropical bases. We recall that a tropical basis for an ideal $I$ is a finite set of polynomials generating the ideal and such that the intersection of the hypersurfaces defined by these polynomials is equal to the tropicalization of the variety defined by $I$.

Every Hilbert polynomial $p$ can be expressed in a canonical way in terms of natural numbers $m_0, m_1, \dots m_s$ (see subsection \ref{subsez:Hilb poly}). Of these numbers, the number $m_0$ is particularly important in the construction of the embedding of the Hilbert scheme. This number also gives a bound on the degree of a tropical basis:

\begin{thm*}{\ref{grado_m0}} Let $I$ be a saturated homogeneous ideal, with Hilbert polynomial~$p$. Then there exist a tropical basis consisting of polynomials of degree at most $m_0$. In particular this bound only depends on $p$.
\end{thm*}

As a corollary of theorem \ref{teo:teorema_principale}, we are able to say that tropical varieties have little dependence on the valued field one chooses for their construction. In particular:

\begin{thm*}{\ref{teo:indipendenza}} The set of tropical varieties definable over an algebraically closed valued field only depends on the characteristic pair of the valued field and on the image of the valuation. If one considers valuations that are surjective on $\erre$, the set of tropical varieties only depends on the characteristic pair.
\end{thm*}

The application that associates a tropical variety to a point of the tropicalization of the Hilbert scheme is surjective (on the set of all tropical varieties that are tropicalizations of algebraic varieties with a fixed Hilbert polynomial) but is not in general injective. To understand why this happens we studied two kinds of examples. 
The first is about the Hilbert schemes of hypersurfaces. Here one problem is that there are several different tropical polynomials that define the same function, and thus the same hypersurface, whenever one of the terms is never the prevalent one. Anyway, it is possible to adjust things so that the map becomes injective. It is necessary to add some extra structure to the tropical hypersurfaces, namely to add weights to the maximal faces, as usual. Once this extra structure is considered, there exists a unique subpolyhedron $P \subset \tropicalization\left(\hilb{p}{n}\right)$ such that the restriction of the correspondence to $P$ is bijective.

We think that this property of the existence of a subpolyhedron that is a ``good'' parameter space should probably be true also for the general Hilbert scheme, but in general it is not clear what is the suitable extra structure. The second kind of examples is about the Hilbert schemes of the pairs of points in the tropical plane. In this case adding the weights to the tropical variety structure is not enough and this is because the tropicalization of the Hilbert scheme seems to remember more information about the nonreduced structure than is expressable by a single integer. This suggests that it would be necessary to enrich the structure of tropical variety even more.

For the convenience of the reader we include here a brief overview of each section.
Section \ref{sez:prelim} opens with some essential facts on the Hilbert polynomial. It then proceeds describing the natural embedding of the Hilbert scheme in the projective space, through the Grassmannian. This is the embedding that will be used for the tropicalization.
Section 3 contains some general facts on valued fields, which are necessary to prove theorem \ref{teo:indipendenza}, along with some fundamental facts about tropical varieties. At the end of the section we state and prove theorem \ref{grado_m0}, although the proof of a technical lemma is postponed to section \ref{sez:bound_monomials}.
The main theorem, \ref{teo:teorema_principale}, and its corollary on the dependence on the base field are stated and proved in section 4.
Section \ref{sez:bound_monomials} itself is completely devoted to prove the technical lemma used in the proof of theorem \ref{grado_m0}. To do so we use the arithmetic degree and primary decomposition to investigate the degree of monomials contained in an ideal.
In the last section we treat our two examples: hypersurfaces and pairs of points in the projective plane.

\section{Preliminaries}\label{sez:prelim}
\subsection{The Hilbert Polynomial}    \label{subsez:Hilb poly}
Let $k$ be a field and $S = k[x_0, \dots, x_n]$, the free graded $k$-algebra of polynomials in $n+1$ variables. We denote by $S_d \subset S$ the vector subspace of homogeneous polynomials of degree $d$. If $I$ is an homogeneous ideal, we denote its homogeneous parts by $I_d = I \cap S_d$. The quotient algebra has a natural grading $S/I = \bigoplus_d {(S/I)}_d = \bigoplus_d S_d / I_d$. The \nuovo{Hilbert function} of $I$ is the function $h_I:\enne \freccia \enne$ defined by:
$$h_I(d) = \dim_k( S_d / I_d ) = \dim_k S_d - \dim_k I_d$$
It is not difficult to see (\cite[Thm. 1.11]{Eis}) that for every homogeneous ideal $I$ there exists a number $d_0$ and a polynomial $p_I$ of degree $s \leq n$, with rational coefficients, such that for $d \geq d_0$ we have
$$h_I(d) = p_I(d)$$  
The polynomial $p_I$ is the \nuovo{Hilbert polynomial} of $I$. Note that $p_I = p_{I^{\sat}}$, hence the Hilbert polynomial only depends on the subscheme of $\pro^n$ defined by $I$. Many invariants of this subscheme may be read from $p_I$, for example $\dim(I) = s = \deg(p_I)$, and $\deg(I) = s! a_s$, where $a_s x^s$ is the term of higher degree of $p_I$.

A \nuovo{numerical polynomial} is a polynomial with rational coefficients taking integer values for all large enough integer arguments. This includes the Hilbert polynomials of homogeneous ideals. Given a finite sequence of numbers $m_0, \dots, m_s \in \ze$, with $m_s \neq 0$, the following formula defines an integer polynomial in $x$ of degree $s$ with term of higher degree $\frac{m_s}{s!} x^s$:
$$g(m_0, \dots, m_s; x) = \sum_{i=0}^s \binom{x+i}{i+1} - \binom{x+i-m_i}{i+1}$$
Moreover every numerical polynomial of degree $s$ can be expressed in the form $g(m_0, \dots, m_s; x)$ (see \cite[Lemma 1.3]{Ba}).

There is a simple description of which numerical polynomials are the Hilbert polynomial of some homogeneous ideal: Let $m_0, \dots, m_s \in \ze$ with $m_s \neq 0$. Then there exists a non-zero homogeneous ideal $I \subset S$ with $p_I = g(m_0, \dots, m_s; x)$ if and only if $s < n$ and $m_0 \geq \dots \geq m_s > 0$ (see \cite[Cor. 5.7]{Har66}).

Some examples. If $I = (0)$, then $p_I(x) = g(m_0,\dots,m_n; x)$ with $m_0 = \dots = m_n = 1$ (see \cite[Chap. 2, 1.15]{Ba}). If $I = (f)$ is a principal ideal with $\deg(f) = d$, then $p_I(x) = g(m_0,\dots,m_{n-1}; x)$ with $m_0 = \dots = m_{n-1} = d$ (see \cite[p. 30]{Ba}). At the other extreme, if $\dim(I) = 0$, with $\deg(I) = d$, then $p_I(x) = g(d; x)=d$. If $I$ defines a curve of degree $d$ and arithmetic genus $g$, then $p_I(x) = g(\binom{d}{2}+1-g, d; x)$ (see \cite[Chap. 2, 1.17]{Ba}). The other easy case is when $I$ is generated by linear forms, and $\dim(I) = s$. Then $p_I(x) = g(m_0,\dots,m_s; x)$ with $m_0 = \dots = m_s = 1$.

The \nuovo{Castelnuovo-Mumford regularity} of a saturated homogeneous ideal $I$, denoted by $\reg(I)$, is the smallest integer $m$ such that $I$ is generated in degree not greater than $m$ and, for every $i$, the $i$-th syzygy module of $I$ is generated in degree not greater than $m+i$ (see \cite[sec. 20.5 and ex. 20.20]{Eis}, \cite[Chap. 2, 2.1 and Lemma 2.4]{Ba}). For every $d \geq \reg(I) - 1$ we have $h_I(d) = p_I(d)$ (see \cite[Chap. 2, 2.5]{Ba}). 

For every saturated ideal $I$ with Hilbert polynomial $g(m_0,\dots,m_s;x)$ we have $\reg(I) \leq m_0$ (see \cite[Chap. 2, Prop. 9.4 and Prop. 10.1]{Ba}). This fact is fundamental in the construction of the projective embedding of the Hilbert Scheme, see below.

To understand the equations of the projective embedding of the Hilbert Scheme, the following fact is needed. Given $s < n$ and $m_0 \geq \dots \geq m_s > 0$, let $I \subset S$ be any homogeneous ideal (possibly non saturated, with any Hilbert polynomial), and choose $m \geq m_0$. If $h_I(x) \leq g(m_0,\dots,m_s; x)$ for $x = m$, then this holds for every $x \geq m$. If $h_I(x) = g(m_0,\dots,m_s; x)$ for $x = m, m+1$, then this holds for every $x \geq m$ (see \cite[Chap. 2, Prop. 9.5 and Prop. 10.2]{Ba}).

\subsection{The embedding of the Hilbert Scheme}\label{subsez:immersion}

The \nuovo{Grassmannian} \G{n}{d} is the set of all the $d$-dimensional vector subspaces of the vector space $k^n$. Equivalently, the Grassmannian can be regarded as the set of the $(d-1)$-dimensional projective subspaces of $\p{}(k^n)$.

The Grassmannian can be embedded naturally in a projective space via the Pl\"ucker embedding. Let $L\subseteq k^n$ be a $d$-dimensional vector subspace and let $v_1,\dots,v_d$ be a basis of $L$. The embedding can be defined as follows:
\[
\begin{array}{ccc}
 \G{n}{d} & \hookrightarrow & {\mathbb P}\left(\bigwedge^d k^n\right) \\
 \langle v_1,\dots,v_d\rangle & \longmapsto & [v_1\wedge\dots\wedge v_d]
\end{array}
\]
It is easy to see that, if one considers two different bases for $L$, the wedge products of the elements of each base will only differ by a multiplicative constant, meaning that this map is well defined in the projective space. Moreover this map is injective, because knowing $v_1\wedge\dots\wedge v_d$ you can write equations for $\langle v_1,\dots,v_d\rangle$. This gives an embedding of $\G{n}{d}$ into ${\mathbb P}\left(\bigwedge^d k^n\right) = {\mathbb P}\left(k^{\binom{n}{d}}\right)$. The image of this embedding is the set of alternate tensors of rank $1$. Equations for this subset are given by the Pl\"ucker ideal, an ideal generated by quadratic polynomials with integer coefficients. This identifies the Grassmannian with a projective algebraic variety.

The Grassmannian can be used to construct the embedding of the Hilbert scheme that we will use later for its tropicalization.

Fix a numerical polynomial $p(x):=g(m_0, \dots, m_s;x)$ and a projective space \p{n}, such that $s < n$ and $m_0 \geq \dots \geq m_s > 0$. It is possible to parametrize all closed subschemes of \p{n} having Hilbert polynomial $p$ with the closed points of a scheme \hilb{p}{n}. This scheme is called the \nuovo{Hilbert scheme} and is a fine moduli space for the closed subschemes of \p{n} up to identity.

Consider the number $m_0$. By the properties stated above, we have that $I = \langle I_{m_0}\rangle^{\sat}$, hence $I$ is determined by its component of degree $m_0$.
$S_{m_0}$ is a vector space of dimension $N:=\binom{n+m_0-1}{n-1}$ and $I_{m_0}$ is a vector subspace of codimension $p(m_0)$. We can define the injective map
\[
\begin{array}{ccc}
\left\{I \subset S \ |\ I \mbox{ saturated ideal, } p_I = p \right\} & \hookrightarrow & \G{N}{N-p(m_0)}\\
I & \mapsto & I_{m_0}
\end{array}
\]
To understand the set-theoretical image of this map, we can use the last sentence of the previous subsection. We have to consider only the ideals $I \subset S$ such that $\dim(I_{m_0}) = N - p(m_0)$ and $\dim(I_{m_0+1}) \leq \binom{n+m_0}{n-1} - p(m_0+1)$. Expressed in the Pl\"ucker coordinates, these relations give equations for the image that are polynomials with integer coefficients, see \cite[Chap. 6, 1.2]{Ba}. Hence we have identified the Hilbert scheme $\hilb{p}{n}$ with a projective algebraic subvariety of the Grassmannian \G{N}{N-p(m_0)}.

Composing this embedding with the Pl\"ucker embedding allows us to embed the Hilbert scheme in a projective space as follows: given an homogeneous saturated ideal $I$ with Hilbert polynomial $p$ let $v_1,\dots,v_d$ be a basis of $I_{m_0}$ as a vector space. Then the point of \hilb{p}{n} corresponding to $I$ is sent into $[v_1\wedge\dots\wedge,v_d]$.

\section{Tropical Varieties}

\subsection{Valued fields}

Let $\cappa$ be an algebraically closed field, with a real valued valuation $v:\cappa^* \mapsto \erre$. We will always suppose that $1$ is in the image of the valuation. 

As $\cappa$ is algebraically closed, the image of $v$ is a divisible subgroup of $\erre$, that we will denote by $\Lambda_v$. This subgroup contains $\qu$, hence it is dense in $\erre$. Moreover, the multiplicative group $\cappa^*$ is a divisible group, hence the homomorphism $v:\cappa^* \mapsto \Lambda_v$ has a section, i.e. there exists a group homomorphism $t^{\bullet}:\Lambda_v \freccia \cappa^*$ such that $v(t^w)=w$ (see \cite[lemma 2.1.13]{MS}). Note that $t^{\bullet}$ satisfies $t^{w+w'} = t^w t^{w'}$. We will denote $t^1$ by $t \in \cappa$. 

We denote by $\ocors = \{x\in\cappa \ |\ -v(x) \leq 0\}$ the \nuovo{valuation ring} of $\cappa$, and by $\mcors = \{x\in\cappa \ |\ -v(x) < 0\}$ the unique maximal ideal of $\ocors$. The quotient $\Delta = \ocors / \mcors$ is the \nuovo{residue field} of $\cappa$, an algebraically closed field. The \nuovo{characteristic pair} of the valued field $\cappa$ is the pair of numbers $\mbox{cp}(\cappa,v) = (\mbox{char}(\cappa), \mbox{char}(\Delta))$. This pair can assume the values $(0,0)$, $(0,p)$ or $(p,p)$, where $p$ is a prime number.

In the following we will need to consider valued fields such that the image of the valuation is surjective into $\erre$. Examples of such fields are provided by the fields of \nuovo{transfinite Puiseux series}. If $\Delta$ is an algebraically closed field, the set
$$\Delta((t^\erre)) = \left\{ \sum_{w \in \erre} a_w t^w \ |\ \{w \in \erre \ |\ a_w \neq 0\} \mbox{ is well ordered }\right\} $$
endowed with the operations of sum and product of formal power series is an algebraically closed field, with surjective real valued valuation $v(\sum_{w \in \erre} a_w t^w) = \min \{w \in \erre \ |\ a_w \neq 0\}$. The residue field of $\Delta((t^\erre))$ is $\Delta$, and the formal monomials $t^w$ form the section $t^\bullet$.

The remainder of this subsection is devoted to prove a proposition that will be needed in subsection \ref{subsez:dependence}. We need to construct some small valued fields first. If $k$ is a field, consider the field $k(t)$ of rational functions in one variable. There is only one real valued valuation on this field that is zero on $k^*$ and such that $v(t) = 1$ (see \cite[Thm. 2.1.4]{EP}). The image of this valuation is $\ze$ and the residue field is $k$. 

In a similar way it is possible to see that every valuation of $\qu$ is either zero on $\qu^*$, or it is the $p$-adic valuation, for a prime number $p$ (see again \cite[Thm. 2.1.4]{EP}). The image of this valuation is $\ze$ and the residue field is $\effe_p$. 

These fields are not algebraically closed, hence we need to extend the valuation to their algebraic closure. By \cite[Thm. 3.1.1]{EP} you see that every valuation of a field admits an extension to the algebraic closure, and by \cite[Thm.3.4.3]{EP} you see that this extension is again real valued. There can be many different ways to construct this extension, but any two of them are related by an automorphism of the algebraic closure fixing the base field $k(t)$ or $\qu$, see \cite[Thm. 3.2.15]{EP}. Hence, up to automorphisms, there exists a unique real valued valuation of the field $\overline{k(t)}$ that is zero on $k^*$ and such that $v(t) = 1$. This valuation has image $\qu$, residue field $\overline{k}$ and characteristic pair $(\mbox{char}(k), \mbox{char}(k))$. In the same way, up to automorphisms, there exists a unique valuation of $\overline{\qu}$ that restricts to the $p$-adic valuation on $\qu$. This valuation has image $\qu$, residue field $\overline{\effe_p}$ and characteristic pair $(0,p)$. 

We will denote by $\effe_{(0,0)}$ the field $\overline{\qu(t)}$ with the valuation described above, by $\effe_{(p,p)}$ the field $\overline{\effe_p(t)}$ with the valuation described above, and by $\effe_{(0,p)}$ the field $\overline{\qu}$ with the $p$-adic valuation. These valued fields are, in some sense, universal, as we show in the following proposition:

\begin{prop}    \label{prop:prime field}
Let $\cappa$ be an algebraically closed field, with a real valued valuation $v:\cappa^* \mapsto \erre$. Then $\cappa$ contains a subfield $\effe$ such that $(\effe, v|_{\effe})$ is isomorphic as a valued field to one of the fields $\effe_{(0,0)}, \effe_{(p,p)}, \effe_{(0,p)}$. 
\end{prop}
\begin{proof}
If $\mbox{char}(\cappa) = p$, it contains a copy of $\effe_p$. The valuation is zero on $\effe_p^*$ because the multiplicative group $\effe_p^*$ is cyclic. Choose an element $t$ such that $v(t) = 1$. Then $\cappa$ contains a copy of $\overline{\effe_p(t)}$. By the arguments above, the valuation restricted to this field is isomorphic to $\effe_{(p,p)}$.

If $\mbox{char}(\cappa) = 0$, it contains a copy of $\qu$. The valuation, restricted to $\qu^*$, can be zero or a $p$-adic valuation. If it is zero, choose an element $t$ such that $v(t) = 1$. Then $\cappa$ contains a copy of $\overline{\qu(t)}$, and, by the arguments above, the valuation restricted to this field is isomorphic to $\effe_{(0,0)}$. If it is the $p$-adic valuation, then $\cappa$ contains a copy of $\overline{\qu}$, and, by the arguments above, the valuation restricted to this field is isomorphic to $\effe_{(0,p)}$. 
\end{proof}

\subsection{Tropical varieties}\label{subsez:tropicalization}

We denote by $\trop = (\erre,\oplus,\odot)$ the \nuovo{tropical semifield}, with tropical operations $a \oplus b = \max(a,b)$ and $a \odot b = a + b$.

Let $\cappa$ be an algebraically closed field, with a real valued valuation $v:\cappa^* \mapsto \erre$, with image $\Lambda_v$. We often prefer to use the opposite function, and we will call it the \nuovo{tropicalization}:
$$
\begin{array}{cccc}
\tau: &  \cappa^*     &  \longmapsto   &    \trop  \\
      &  x            &  \longmapsto   &    -v(x)
\end{array}
$$

The \nuovo{tropicalization map} on $ {(\cappa^*)}^n$ is the component-wise tropicalization function:
$$
\begin{array}{cccc}
\tau: &  {(\cappa^*)}^n     &  \longmapsto   &    \trop^n  \\
      &  (x_1, \dots, x_n)  &  \longmapsto   &    (\tau(x_1), \dots, \tau(x_n))
\end{array}
$$
The image of this map is $\Lambda_v^n$, a dense subset of $\erre^n$.

Let $I \subset \cappa[x_1, \dots, x_n]$ be an ideal, and let $Z = Z(I) \subset \cappa^n$ be its zero locus, an affine algebraic variety. The image of $Z \cap {(\cappa^*)}^n$ under the tropicalization map will be denoted by $\tau(Z)$ and it is a closed subset of $\Lambda_v^n$. The \nuovo{tropicalization} of $Z$ is the closure of $\tau(Z)$ in $\erre^n$, and it will be denoted by $\tropicalization(V)$. In particular, if $v$ is surjective ($\Lambda_v = \erre$) we have $\tau(V) = \tropicalization(V)$. A \nuovo{tropical variety} is the tropicalization of an algebraic variety. Note that in this way the notion of tropical variety depends on the choice of the valued field $\cappa$. We will show in subsection \ref{subsez:dependence} that it actually depends only on the characteristic pair of $\cappa$ and on the value group $\Lambda_v$.  

A \nuovo{tropical polynomial} is a polynomial in the tropical semifield, an expression of the form:
$$\phi = \bigoplus_{a \in \enne^n} \phi_a \odot x^{\odot a} = \max_{a \in \enne^n} \left( \phi_a + \langle x, a \rangle \right)$$
where $\phi_a \in \trop \cup \{-\infty\}$, there are only a finite number of indices $a \in \enne^n$ such that $\phi_a \neq -\infty$. Note that here $x = (x_1, \dots, x_n)$ is a vector of indeterminates. If all the coefficients $\phi_a$ are $-\infty$, then $\phi$ is the null polynomial. If $\phi$ is a non null tropical polynomial, the \nuovo{tropical zero locus} of $\phi$ is the set
$$T(\phi) = \left\{ \omega \in \trop^n \ |\ \mbox{ the max. in } \max_{a \in \enne^n} \left( \phi_a + \langle x, a \rangle \right) \mbox{ is attained at least twice} \right\} $$
If $\phi$ is the null polynomial, then $T(\phi) = \trop^n$.

If $f = \sum_{a \in \enne^n} f_a x^a \in \cappa[x_1, \dots, x_n]$ is a polynomial, the \nuovo{tropicalization} of $f$ is the tropical polynomial
$$\tau(f) = \bigoplus_{a \in \enne^n} \tau(f_a) \odot x^{\odot a} $$
where $\tau(0) = -\infty$.

The fundamental theorem of tropical geometry (see \cite[Thm. 3.2.4]{MS}) states that if $I  \subset \cappa[x_1, \dots, x_n]$ is an ideal, then
$$\tropicalization(Z(I)) = \bigcap_{f \in I} T(\tau(f))$$ 

It is also possible to find some finite systems of generators $f_1, \dots, f_r$ of the ideal $I$ such that 
$$\tau(Z(I)) = T(\tau(f_1)) \cap \dots \cap  T(\tau(f_r))$$ 
Such a system of generators is called a \nuovo{tropical basis}. It is possible to construct tropical bases with few polynomials of high degree, as in \cite{HT}, but we also show that with the methods similar to the ones of \cite[Thm. 2.9]{BJSST} it is possible to construct tropical bases made up of polynomials of controlled degree, see theorem \ref{grado_m0}.  

The following proposition will be used later in subsection \ref{subsez:dependence}.

\begin{prop}    \label{prop:extension}
Let $(\cappa,v)$ be an algebraically closed valued field as above, let $\effe \subset \cappa$ be an algebraically closed subfield such that $1\in v(\effe)$, and consider the valued field $(\effe,v|_\effe)$. Let $I \subset \effe[x_1, \dots, x_n]$ be an ideal, with zero locus $Z = Z(I) \subset \effe^n$. Consider the extension $I^\cappa$ of $I$ to $\cappa[x_1, \dots, x_n]$, with zero locus $Z^\cappa \subset \cappa^n$. Then
$$\tropicalization(Z) = \tropicalization(Z^\cappa) $$
\end{prop}
\begin{proof}
One inclusion ($\tropicalization(Z) \subset \tropicalization(Z^\cappa)$) is obvious because $Z \subset Z^\cappa$. To see the reverse inclusion, note that as $I \subset I^\cappa$, we have  
$$\bigcap_{f \in I} T(\tau(f)) \supset \bigcap_{f \in I^\cappa} T(\tau(f)) $$
and then use the fundamental theorem.
\end{proof}

If $I \subset \cappa[x_0, \dots, x_n]$ is an homogeneous ideal, the variety $Z = Z(I)$ is a cone in $\cappa^{n+1}$: if $x \in Z$ and $\lambda \in \cappa$, then $\lambda x \in Z$. The image of $Z$ in the projective space $\cappa\pro^n$ is a projective variety that we will denote by $V = V(I)$. 

The tropicalization $\tropicalization(Z)$ is also a cone but in the tropical sense: if $\omega \in \tropicalization(Z)$ and $\lambda \in \trop$, then $\lambda \odot \omega = (\lambda +\omega_1, \dots, \lambda + \omega_n) \in \tropicalization(Z)$.

The \nuovo{tropical projective space} is the set $\trop\pro^n = \trop^{n+1} / \sim$, where $\sim$ is the \nuovo{tropical projective equivalence relation}:
\[x \sim y \Leftrightarrow \exists \lambda \in \trop : \lambda \odot x = y \Leftrightarrow \exists \lambda \in \erre : (x_0 + \lambda, \dots, x_n + \lambda) = (y_0, \dots, y_n)\]
We will denote by $[\cdot]:\trop^{n+1} \mapsto \trop\pro^n$ the projection onto the quotient, and we will use the homogeneous coordinates notation: $[x] = [x_0: \dots :x_n]$. To visualize $\trop\pro^n$ it is possible to identify it with a subset of $\trop^{n+1}$. If $i \in \{0, \dots, n\}$, the subset $\{ x \in \trop^{n+1} \ |\ x_i = 0\}$ is the analog of an affine piece, and the quotient map $[\cdot]$, restricted to it, is a bijection. A more invariant way for doing the same thing is to consider the set $\{ x \in \trop^{n+1} \ |\ x_0 + \dots + x_n = 0\}$.

The image of $\tropicalization(Z)$ in the tropical projective space is a \nuovo{tropical projective variety}. There is also a \nuovo{projective tropicalization map}, defined on $\cappa\pro^n \setminus C$, where $C$ is the union of the coordinate hyperplanes:
$$\tau:\cappa\pro^n \setminus C \ni [x] \mapsto [\tau(x_0): \dots : \tau(x_n)] \in \trop\pro^n$$ 
The image of $\tau$ will be denoted by $\pro(\Lambda_v^{n+1})$, a dense subset of $\trop\pro^n$. 
If $V$ is a projective variety in $\cappa\pro^n$, we will denote by $\tau(V)$ the image of $V \cap (\cappa\pro^n \setminus C)$ under the projective tropicalization map, a closed subset of $\pro(\Lambda_v^{n+1})$, and by $\tropicalization(V)$ the closure of $\tau(V)$ in $\trop\pro^n$.

\subsection{Initial ideals}

As before, let $\cappa$ be an algebraically closed field, with a real valued valuation $v:\cappa^* \mapsto \erre$, with image $\Lambda_v$.

Let $f = \sum_{a \in \enne^n} f_a x^a \in \cappa[x_1, \dots, x_n]$ be a non-zero polynomial, and let $\omega = (\omega_1, \dots, \omega_n) \in \Lambda_v^n$ be a vector. Denote by $H$ the number $\tau(f)(\omega)
$. Now the polynomial $t^H f(t^{-\omega_1} x_1, \dots, t^{-\omega_n} x_n) \in \ocors[x_1, \dots, x_n] \setminus \mcors[x_1, \dots, x_n]$. Its image in $\Delta[x_1, \dots, x_n]$ is a non-zero polynomial $\ini_\omega(f)$, called the \nuovo{initial form} of $f$ in $\omega$. If $f = 0$, we put $\ini_\omega(f) = 0$. 
 
Let $I \subset \cappa[x_1, \dots, x_n]$ be an ideal. The initial ideal of $I$ is the set $\ini_\omega(I) = \{\ini_\omega(f) \ |\ f \in I\} \subset \Delta[x_1, \dots, x_n]$.

\begin{prop}
The set $\ini_\omega(I)$ is an ideal of $\Delta[x_1, \dots, x_n]$.
\end{prop}
\begin{proof}
If $f \in \ini_\omega(I)$, it is clear that for every monomial $x^\alpha$, the product $x^\alpha f \in \ini_\omega(I)$. We only need to verify that if $f,g \in \ini_\omega(I)$ then $f+g \in \ini_\omega(I)$. If $f+g = 0 = \ini_\omega(0)$ there is nothing to prove, hence we can suppose that $f+g \neq 0$.

Let $F,G \in I$ such that $\ini_\omega(F) = f$, $\ini_\omega(G) = g$. Up to multiplying by an element of the section $t^\bullet$, we may suppose that $F(t^{\omega_1} x_1, \dots, t^{\omega_n} x_n)$, $G(t^{\omega_1} x_1, \dots, t^{\omega_n} x_n)$ $\in$ $\ocors[x_1, \dots, x_n] \setminus \mcors[x_1, \dots, x_n]$. As $f+g \neq 0$, also $(F+G)(t^{\omega_1} x_1, \dots, t^{\omega_n} x_n)$ $\in$ $\ocors[x_1, \dots, x_n] \setminus \mcors[x_1, \dots, x_n]$. Hence $\ini_\omega(F+G) = f+g$.  
\end{proof}

Note that the definitions we have given of the initial form and initial ideal only work because we suppose that $\omega \in \Lambda_v^n$. If $f$ is a polynomial, $\ini_\omega(f)$ is a monomial if and only if the maximum in $\tau(f)(\omega)$ is attained only once. Therefore tropical varieties can be described in terms of initial ideals: $\tau(V(I))$ is the set of all $\omega \in \Lambda_v^n$ such that $\ini_\omega(I)$ contains no monomials, and $\tropicalization(V(I))$ is the closure of this set.  

\begin{lemma}    \label{lemma:dimension}
Let $L \subset \cappa^n$ be a vector subspace. Denote by $\pi$ the quotient of $\ocors$-modules $\pi:\ocors^n \freccia \ocors^n / \mcors^n = \Delta^n$. Then the image $L' = \pi(L\cap\ocors^n)$ is a vector subspace of $\Delta^n$ with $\dim_\Delta(L') = \dim_\cappa(L)$.
\end{lemma}
\begin{proof}
$L'$ is an $\ocors$-submodule of $\Delta^n$, hence it is a vector subspace. Put $h=\dim_\Delta(L')$. Let $v_1, \dots, v_h$ be a basis of $L'$, and complete it with vectors $v_{h+1}, \dots, v_n$ to a basis of $\Delta^n$. For every $i \leq h$ it is possible to find an element $w_i \in \ocors^n \cap L$ such that $\pi(w_i) = v_i$. For every $i > h$ it is possible to find an element $w_i \in \ocors^n$ such that $\pi(w_i) = v_i$, and $w_i$ is necessarily not in $L$. 

The elements $w_1, \dots, w_n$ are linearly independent over $\cappa$: if $a_1 w_1 + \dots + a_n w_n = 0$, and some of the $a_i$ is not zero, up to multiplying all the $a_i$'s by an element of the section $t^\bullet$, we can suppose that all the $a_i$'s are in $\ocors$, and some of them is not in $\mcors$. Then $\pi(a_1) v_1 + \dots + \pi(a_n) v_n = 0$ is a non-trivial linear combination. 

Let $A = \linspan(w_1, \dots, w_h)$ and $B = \linspan{w_{h+1}, \dots, w_n}$. We know that $A \subset L$. We only need to show that $B\cap L = (0)$, this implies $L = A$.

Let $x \in B\cap L$, $x = a_{h+1} w_{h+1} + \dots + a_n w_n$. If $x \neq 0$, up to multiplying all the $a_i$'s by an element of the section $t^\bullet$, we can suppose that all the $a_i$'s are in $\ocors$, and some of them is not in $\mcors$. Now $\pi(x) = \pi(a_{h+1}) v_{h+1} + \dots + \pi(a_n) v_n \in L'$ because $x \in L$. Hence $\pi(x) = 0$, and this forces all $\pi(a_{h+1}), \dots, \pi(a_n)$ to be zero, which is a contradiction. 
\end{proof}

\begin{prop}
Let $I \subset \cappa[x_0, \dots, x_n]$ be an homogeneous ideal, and let\linebreak $\omega = (\omega_0, \dots, \omega_n) \in \Lambda_v^{n+1}$. Then $\ini_\omega(I)$ is an homogeneous ideal, with the same Hilbert function as $I$. 
\end{prop}
\begin{proof}
The map 
$$\cappa[x_0, \dots, x_n] \ni f(x_0, \dots, x_n) \mapsto f(t^{\omega_1} x_1, \dots, t^{\omega_n} x_n) \in \cappa[x_0, \dots, x_n]$$
is an isomorphism of $\cappa[x_0, \dots, x_n]$. Up to this isomorphism, we can suppose that $\omega = 0$.

If $f \in \ini_0(I)$, there is $F \in$ $\ocors[x_1, \dots, x_n] \setminus \mcors[x_1, \dots, x_n]$ such that $f = \ini_0(F)$. The homogeneous components of $f$ are then the initial forms of the homogeneous components of $F$ not contained in $\mcors[x_1, \dots, x_n]$. Hence $\ini_0(I)$ is an homogeneous ideal.

The ideal $\ini_0(I)$ is the image in $\Delta[x_0, \dots, x_n]$ of $I \cap \ocors[x_0, \dots, x_n]$, via the quotient map $\ocors[x_0, \dots, x_n] \mapsto \ocors[x_0, \dots, x_n]/\mcors[x_0, \dots, x_n] = \Delta[x_0, \dots, x_n]$. For every $d \in \enne$, the homogeneous component of degree $d$, ${(\ini_0(I))}_d$, is the image in ${(\Delta[x_0, \dots, x_n])}_d$ of $I_d \cap {(\ocors[x_0, \dots, x_n])}_d$. By previous lemma, $\dim_\Delta({(\ini_0(I))}_d) = \dim_\cappa(I_d)$.
\end{proof}   

Note that even if $I$ is a saturated homogeneous ideal, $\ini_\omega(I)$ need not be saturated.

\begin{prop}
Let $I \subset \cappa[x_0, \dots, x_n]$ be a saturated homogeneous ideal, with Hilbert polynomial $g(m_0, \dots, m_s; x)$. Then for every $\omega \in \erre^{n+1}$, if the initial ideal $\ini_\omega(I)$ contains a monomial, it also contains a monomial of degree not greater than $m_0$.
\end{prop}
\begin{proof}
If $\ini_\omega(I)$ contains a monomial, then its saturation ${(\ini_\omega(I))}^{\sat}$ also contains a monomial. By lemma \ref{lemma:monomi contenuti} (to be proved in the next subsection), ${(\ini_\omega(I))}^{\sat}$ contains a monomial of degree $m_0$.

As ${(\ini_\omega(I))}^{\sat}$ is saturated, its Hilbert function in degree $m_0$ is equal to its Hilbert polynomial $g(m_0, \dots, m_s; m_0)$ (see discussion in subsection \ref{subsez:Hilb poly}). The Hilbert function of $\ini_\omega(I)$ is equal to the Hilbert function of $I$, that is saturated, hence in degree $m_0$ also this Hilbert function is equal to the Hilbert polynomial $g(m_0, \dots, m_s; m_0)$. In particular, in degree $m_0$, the Hilbert functions of $\ini_\omega(I)$ and its saturated are equal, hence their components in degree $m_0$ coincide. The monomial of degree $m_0$ we found in ${(\ini_\omega(I))}^{\sat}$ is also in $\ini_\omega(I)$.
\end{proof}

Now we use the previous proposition give a bound on the degree of a tropical basis of a saturated ideal in terms of its Hilbert polynomial. The tropical basis is constructed by adapting the methods of \cite[Thm. 2.9]{BJSST} to the ``non-constant coefficients'' case.

\begin{thm}         \label{grado_m0}
Let $I \subset \cappa[x_0, \dots, x_n]$ be a saturated homogeneous ideal, with Hilbert polynomial $g(m_0, \dots, m_s; x)$. Then there exist a tropical basis $f_1, \dots, f_s \in I$ such that for all $i$, $\deg(f_i) \leq m_0$. In particular, if $I_{m_0}$ denotes the component of degree $m_0$ of $I$, we have
$$\tropicalization(Z(I)) = \bigcap_{f \in I_{m_0}} T(\tau(f))$$
\end{thm}
\begin{proof}
When $\omega$ varies in $\Lambda_v^n$, the initial ideals $\ini_\omega(I)$ assume only a finite number of values (see \cite[Thm. 2.2.1]{Spe} or \cite[Thm. 2.4.11]{MS}). We choose $\omega_1, \dots, \omega_p$ such that $\ini_{\omega_1}(I), \dots, \ini_{\omega_p}(I)$ are all the initial ideals containing monomials. 

For every one of these $\omega_i$ we will construct a polynomial $f_i \in I$ in the following way. As $\ini_{\omega_i}(I)$ contains a monomial, by previous proposition we know that it also contains a monomial $x^a$ of degree $m_0$. Choose a basis $g_1, \dots, g_h$ of ${(\ini_{\omega_i}(I))}_{m_0}$ (the component of degree $m_0$) such that $g_1 = x^a$. As the set of all monomials of degree $m_0$ is a basis of ${(\Delta[x_0, \dots, x_n])}_{m_0}$, the independent set $g_1, \dots g_h$ may be extended to a basis of ${(\Delta[x_0, \dots, x_n])}_{m_0}$ by adding monomials $g_{h+1}, \dots g_N$. Now $g_1$ and $g_{h+1}, \dots g_N$ may also be interpreted as monomials in ${(\cappa[x_0, \dots, x_n])}_{m_0}$. As they are monomials, for every $\omega \in \Lambda_v^n$ we have $\ini_\omega(g_j) = g_j$ (for $j = 1$ or $j > h$). By reasoning as in the proof of lemma \ref{lemma:dimension}, the elements $g_{h+1}, \dots, g_N$ give a basis of ${(\cappa[x_0, \dots, x_n])}_{m_0} / I_{m_0}$, hence there is a unique expression 
$$g_1 = f_i + \sum_{j>h} c_j g_j$$
where $f_i \in I_{m_0}$. $f_i$ is the polynomial we searched for, its construction only depends on $\ini_{\omega_i}(I)$, and not on the particular value of $\omega_i$ realizing this initial ideal. Now consider any $\omega$ such that $\ini_{\omega}(I) = \ini_{\omega_i}(I)$. We want to see that $\ini_{\omega}(f_i) = g_1 = x^a$. This is because $\ini_\omega(f_i)$ must be a linear combination of $g_1$ and the $g_j$ with $j > h$. But the $g_j$ with $j > h$ form a basis of a vector subspace of ${(\Delta[x_0, \dots, x_n])}_{m_0}$ that is complementary to $\ini_\omega(I)$, while $\ini_\omega(f_i)$ is in $\ini_\omega(I)$. Hence $\ini_\omega(f_i) = g_1$.

Now that we have constructed the polynomials $f_1, \dots, f_p$ we add to them other polynomials $f_{p+1}, \dots, f_s$ such that they generate the ideal $I$. We have to prove that, if $T = T(\tau(f_1)) \cap \dots \cap T(f_s)$, we have 
$$\tropicalization(Z(I)) = T$$
The inclusion $\tropicalization(Z(I)) \subset T$ is clear. To show the reverse inclusion, first note that $T$  is a finite intersection of tropical hypersurfaces, hence it is a $\Lambda_v$-rational polyhedral complex. In particular, $T \cap \Lambda_v^n$ is dense in $T$. For this reason we only need to verify that every $\omega \in T \cap \Lambda_v^n$ is in $\tropicalization(Z(I))$, or, conversely, that if $\omega \in \Lambda_v^n$ is not in $\tropicalization(Z(I))$, then it is not in $T$. 

We have that $\ini_\omega(I) = \ini_{\omega_i}(I)$ for some $i$. Then the polynomial $f_i$ has the property that $\ini_\omega(f_i)$ is a monomial. Hence $\omega$ is not contained in $T(\tau(f_i))$.
\end{proof}

Note that the technique used in the previous proof to construct the polynomials $f_i$ is similar to the division algorithm in standard Gr\"obner bases theory. The difference is that in this case the initial ideals $\ini_\omega(I)$ are not, in general, monomial ideals hence the monomials not contained in $\ini_\omega(I)$ does not form a basis for a complementary subspace. To overcome this problem we had to choose a subset of those monomials forming a basis for a complementary subset, hence the ``remainder'' of the division is not canonically determined, but it depends on this choice.

\section{Applications}

\subsection{Tropicalization of the Hilbert Scheme}

Let $\cappa$ be an algebraically closed field, with a real valued valuation $v:\cappa^* \mapsto \erre$. Denote by $\Lambda_v$ the image of $v$, as above. Fix a projective space $\cappa\pro^n$, and consider a numerical polynomial $p(x):=g(m_0,\dots,m_s;x)$ such that $s < n$ and $m_0 \geq \dots \geq m_s > 0$. To simplify notation, let $N:=\binom{n+m_0-1}{n-1}$, and $M = \binom{N}{p(m_0)}$. Using the embedding described in subsection \ref{subsez:immersion}, we can identify the Hilbert scheme \hilb{p}{n} with an algebraic subvariety of the projective space $\cappa\pro^{M-1}$,
contained in the Pl\"ucker embedding of the Grassmannian \G{N}{N-p(m_0)}.   

We construct the tropicalization of the Hilbert scheme using this embedding. As usual we denote by $\tau\left(\hilb{p}{n} \right) \subset \pro\left(\Lambda^{M}\right)$ the image of the Hilbert scheme via the tropicalization map, and by $\tropicalization\left(\hilb{p}{n} \right)$ its closure in $\trop\pro^{M-1}$. Of course, if $\Lambda_v = \erre$, we have $\tau\left(\hilb{p}{n} \right) = \tropicalization\left(\hilb{p}{n} \right)$.

For every point $x \in \hilb{p}{n}$, denote by $V_x \subset \cappa\pro^n$ the algebraic subscheme parametrized by $x$. 

\begin{thm}
Let $x,y \in \hilb{p}{n} \subset \cappa\pro^{M-1}$. If $\tau(x) = \tau(y)$, then $\tropicalization(V_x) = \tropicalization(V_y)$.
\end{thm}
\begin{proof}
Let $I$ and $J$ be the ideals corresponding, respectively, to $V_x$ and $V_y$. The homogeneous parts $I_{m_0}, J_{m_0}$, considered as vector subspaces of $S_{m_0}$, corresponds to points of the Grassmannian \G{N}{N-p(m_0)}. By \cite[Thm. 3.8]{SS} if two points of the Grassmannian have the same tropicalization, the linear spaces the parametrize have the same tropicalization. As we know that $\tau(x) = \tau(y)$, we have that $\tropicalization(I_{m_0}) = \tropicalization(J_{m_0})$. More explicitly, this means that
$$\{\tau(f) \ |\ f \in I_{m_0}\} = \{\tau(f) \ |\ f \in J_{m_0}\} $$
By theorem \ref{grado_m0}, this implies that $\tropicalization(V_x) = \tropicalization(V_y)$.
\end{proof}

We can sum up what we said so far in the following statement showing that the tropicalization of the Hilbert scheme can be interpreted as a parameter space for the set of tropical varieties that are the tropicalization of subschemes with a fixed Hilbert polynomial. 

\begin{thm}
\label{teo:teorema_principale}
There is a commutative diagram
\[
\begin{tikzpicture}[node distance=6cm]
 \node (G) {\hilb{p}{n}};
 \node[right of=G] (pl) {$\left\{ V \subset \cappa\pro^n \ |\ V \mbox{ has Hilbert polynomial } p \right\}$};
  \draw[<->] (G)-- node[auto] {$b$} (pl);
 \node[below=1cm of G] (TG) {$\tau(\hilb{p}{n})$};
 \node[right of=TG] (Tpl) {$\left\{ \tropicalization(V) \subset \trop\pro^n \ |\ V \mbox{ has Hilbert polynomial } p \right\}$};
  \draw[->]  (G)   -- node[auto]{$\tau$} (TG);
  \draw[->>] (TG)  -- node[auto]{$s$}                   (Tpl);
  \draw[->]  (pl)  -- node[auto]{$\tropicalization$} (Tpl);
\end{tikzpicture}
\]
Where $b$ is the classical correspondence between points of the Hilbert scheme and subschemes of $\cappa\pro^n$, and $s$ is defined as
\[s(x) = \bigcap_{f\in L(x)}T(f)\]
where $L(x)$ is the tropical linear subspace of the space of tropical polynomials associated to the point $x$, seen as a point of the tropical Grassmannian.
Moreover, the map $s$ is surjective.
\end{thm}

This parametrization is surjective, but it is in general not injective, as we will see in section \ref{sez:examples}. This result extends to the Hilbert schemes a property that was already known for the tropical Grassmannian (see \cite[Thm. 3.8]{SS}).

\subsection{Dependence on the valued field}   \label{subsez:dependence}

The definition of tropical variety, as it was given in subsection \ref{subsez:tropicalization}, depends on the choice of a valued field $\cappa$. For an example of this, see \cite[Thm. 7.2]{SS}, where it is exhibited a tropical linear space defined using a field of characteristic $2$ that cannot be obtained as tropicalization of a linear space over a field of characteristic $0$. 

Using the construction of transfinite Puiseaux series, one can construct larger and larger valued fields, and it is \emph{a priori} possible to think that if the field is sufficiently large, new tropical varieties may appear. As a corollary of our result about the Hilbert scheme, we show that that is not the case: the definition of tropical variety only depends on the characteristic pair of the valued field, and on the image group $\Lambda_v$. If we restrict our attention only to the largest possible image group, $\erre$, we have that the definition of tropical variety only depends on the characteristic pair of the field.

\begin{thm}
Let, $(\cappa,v), (\cappa',v')$ be two algebraically closed fields, with real valued valuations $v:\cappa^* \mapsto \erre$, $v':{\cappa'}^* \mapsto \erre$. Suppose that $\cp(\cappa) = \cp(\cappa')$, and $\Lambda_v \subset \Lambda_{v'}$. Let $V \subset \cappa\pro^n$ be a subscheme. Then there exists a subscheme $W \subset \cappa'\pro^n$ with the same Hilbert polynomial as $V$ and such that $\tropicalization(V) = \tropicalization(W)$. 
\end{thm} 
\begin{proof}
Let $p$ be the Hilbert polynomial of $V$. Consider the Hilbert scheme $\hilb{p}{n}$. We claim that the tropical variety $\tropicalization(\hilb{p}{n})$ is the same for the two fields $\cappa$ and $\cappa'$. This is because, by proposition \ref{prop:prime field} both fields contain a subfield $\effe$ isomorphic to one of the fields $\effe_{(0,0)}, \effe_{(p,p)}, \effe_{(0,p)}$. As we said in subsection \ref{subsez:immersion}, the Hilbert scheme is defined by equations with integer coefficients, equations that are not dependent on the field. By proposition \ref{prop:extension}, $\tropicalization(\hilb{p}{n})$ is equal if defined over $\effe$ or $\cappa$ or $\cappa'$. 

Now if $x \in \hilb{p}{n}$ (over $\cappa$) is the point corresponding to $V$, we have that $\tau(x) \in \tropicalization(\hilb{p}{n}) \cap \pro(\Lambda_v^{M})$. As $\Lambda_v \subset \Lambda_{v'}$, there is a point $y \in \hilb{p}{n}$ (over $\cappa'$) such that $\tau(y) = \tau(x)$. Now let $W \subset \cappa'\pro^n$ be the subscheme corresponding to $y$. By theorem \ref{teo:teorema_principale}, we have $\tropicalization(W) = \tropicalization(V)$.
\end{proof}

This theorem can be restated as follows:

\begin{thm}\label{teo:indipendenza}
The set of tropical varieties definable over an algebraically closed valued field $(\cappa,v)$ only depends on the characteristic pair of $(\cappa,v)$, and on the image group $\Lambda_v$. If you suppose $\Lambda_v = \erre$, the set of tropical varieties only depends on the characteristic pair of $(\cappa,v)$.
\end{thm}

\section{Existence of monomials of bounded degree}\label{sez:bound_monomials}

The aim of this section is to prove lemma \ref{lemma:monomi contenuti}, that was needed above.

\subsection{Arithmetic Degree}   \label{subsez:arit deg}

Let $I \subset S$ be an homogeneous ideal. A \nuovo{primary decomposition} of $I$ is a decomposition
$$I = \bigcap_{i = 1}^k Q_i$$
where the $Q_i$'s are homogeneous primary ideals, no $Q_i$ is contained in the intersection of the others, and the radicals of the $Q_i$'s are distinct. The radicals $P_i = \sqrt{Q_i}$ are called prime ideals \nuovo{associated} with $I$, and they don't depend on the choice of the decomposition (see \cite[p. 51, Thm. 4.5, Thm. 7.13]{AM}, \cite[Sect. 3.5]{Eis}).

For every $i \in \{1\dots k\}$ let 
$$I_i = \bigcap \{Q_j | P_j \subsetneq P_i \}$$
If $P_i$ is an isolated component, then $I_i = (1)$, while if $P_i$ is embedded, then $I_i \subset P_i$. Note that, even if the primary ideals $Q_j$ are not always univocally determined by $I$, the ideals $I_i$ and $I_i \cap Q_i$ are univocally determined, see \cite[Thm. 4.10]{AM}. The irrelevant ideal $(x_0, \dots, x_n)$ is associated with $I$ if and only if $I$ is not saturated. In this case, if $(x_0, \dots, x_n) = P_1$, then the saturation $I^{\sat}$ is equal to the ideal $I_1$.

Following \cite[p. 27]{BM92}, for every $i \in \{1\dots k\}$ we define the multiplicity of $P_i$ in $I$ (written $\mult_I(P_i)$) as the length $\ell$ of a maximal chain of ideals
$$Q_i \cap I_i = J_\ell \subsetneq J_{\ell-1} \subsetneq \dots \subsetneq J_1 = P_i \cap I_i$$
where every $J_j = R_j \cap I_i$ for some $P_i$-primary ideal $R_j$.

For every $r \in \{-1, \dots, n-1\}$, the \nuovo{arithmetic degree} of $I$ in dimension $r$ is
$$\adeg_r(I) = \sum_{\{i \ |\ \dim(P_i) = r\}}  \mult_I(P_i) \deg(P_i)$$
If $r > \dim(I) = \max \dim(P_i)$, then $\adeg_r(I)=0$. If $s = \dim(I)$, then $\adeg_s(I)=\deg(I)$. For every $0 \leq r \leq \dim(I)$ we have $\adeg_r(I) = \adeg_r(I^{\sat})$. The only prime ideal of dimension $-1$ is the irrelevant ideal, which has degree $1$. The arithmetic degree in dimension $-1$ indicates how much $I$ is non saturated: $\adeg_{-1}(I) = \dim_k(I^{\sat}/I)$.

The \nuovo{arithmetic degree} of $I$ is
$$ \adeg(I) = \sum_{i=0}^n \adeg_r(I) $$
Note that the value $\adeg_{-1}(I)$ does not appear in the arithmetic degree, in particular $\adeg(I) = \adeg(I^{\sat})$.

\begin{thm}
Let $I$ be an homogeneous ideal and let $g = g(m_0, \dots, m_s; x)$ be its  Hilbert polynomial. Then for every $i \in \{0,\dots,s\}$:
$$\sum_{r=i}^s \adeg_r(I) \leq m_i$$
In particular 
$$\adeg(I) \leq m_0$$ 
Moreover these bounds are optimal: For every Hilbert polynomial $g$, there exists an ideal $I$ with Hilbert polynomial $g$ and such that for all $i \in \{0,\dots,s\}$ we have $\sum_{r=i}^s \adeg_r(I) = m_i$.
\end{thm}
\begin{proof}
This fact can be deduced by putting together a few results from \cite{Har66}.

For every $r\geq 0$, $\adeg_r(I) = n_r(X)$, where $X$ is the closed subscheme of $\p{n}$ defined by $I$ and $n_r(X)$ is defined in \cite[p. 21, Rem. 3]{Har66}. See \cite[p. 420]{STV95} for details. 

A tight fan is a particular kind of closed subscheme of $\p{n}$, see the definition at the beginning of \cite[Chap. 3]{Har66}. If $X$ is a tight fan with Hilbert polynomial $g$, then $\sum_{r=i}^s n_r(X) = m_i$ (see \cite[Cor. 3.3]{Har66}).  

Given any closed subscheme $X$ of $\p{n}$ with Hilbert polynomial $g$, it is possible to construct a tight fan $Y$ with Hilbert polynomial $g$ and $n_r(X) \leq n_r(Y)$ (as in the proof of \cite[Thm. 5.6]{Har66}). Hence $\sum_{r=i}^s n_r(X) \leq \sum_{r=i}^s n_r(Y) = m_i$.
\end{proof}

\subsection{Proof of lemma \ref{lemma:monomi contenuti} }

We will use the following standard notation. Let $I \subset S$ be an homogeneous ideal and $f \in S$. 

$$(I:f) = \{g \in S \ |\ gf \in I\}$$
$$(I:f^\infty) = \{g \in S \ |\ \exists n: g f^n \in I\}$$

See \cite[Ex. 1.12]{AM} for some properties of $(I:f)$, for example $(I:fg)=((I:f):g)$ and $(\bigcap_i I_i : f) = \bigcap_i (I_i : f)$. This same properties also hold for $(I:f^\infty)$: $(I:fg^\infty)=((I:f^\infty):g^\infty)$ and $(\bigcap_i I_i : f^\infty) = \bigcap_i (I_i : f^\infty)$. 

Let $I  = \bigcap_{i = 1}^k Q_i$ be a primary decomposition of $I$, as above, with $\sqrt{Q_i} = P_i$. Up to reordering, we can suppose that $f \in P_1, \dots, P_h$ and $f \not\in P_{h+1}, \dots, P_k$.

\begin{lemma}     \label{lemma:iterated division}
$$(I:f^\infty) = \bigcap_{i=h+1}^k Q_i$$
\end{lemma}   
\begin{proof}
$$(I : f^\infty) = (\bigcap_{i=1}^k Q_i : f^\infty) = \bigcap_{i=1}^k ( Q_i : f^\infty) $$
If $i \leq h$, we have $f \in P_i$, hence $f^m \in Q_i$ for some $m$, hence $(Q_i:f^m) = (1) = ( Q_i : f^\infty)$.

If $i > h$, we have $f \not\in P_i$, hence $(Q_j : f) = Q_j$ by \cite[Lemma 4.4]{AM}, and $(Q_j : f^\infty) = Q_j$.   
\end{proof}

We will need a way for estimating multiplicities, and to do that we need to construct some strictly ascending chains. 

\begin{lemma}  \label{lemma:strict chains}
Consider ideals $J,Q \subset S$ and elements $f \in S$ and $l \in \enne$. The following statements are equivalent:
\begin{enumerate}
\item There exists $a \in J$ such that $af^l \in Q$ and $af^{l-1} \not\in Q$.
\item The chain 
$$Q \cap J \subset (Q:f) \cap J \subset \dots \subset (Q:f^l) \cap J$$
is strictly ascendant.
\item $(Q:f^{l-1}) \cap J \subsetneq (Q:f^l) \cap J$.
\end{enumerate}
\end{lemma}
\begin{proof}
$(1) \Rightarrow (2)$: For all $i \in \{1,\dots,l\}$, $af^{l-i} \in (Q:f^i) \cap J$, but $af^{l-i} \not\in (Q:f^{i-1}) \cap J$.

$(2) \Rightarrow (3)$: Trivial. 

$(3) \Rightarrow (1)$: Let $a$ be any element of $(Q:f^l) \cap J \setminus (Q:f^{l-1}) \cap J$. Then $a \in J$, $af^l \in Q$, but $af^{l-1} \not\in Q$.
\end{proof}

We will denote by $\ell_f(I)$ the minimum $l \in \enne$ such that $(I:f^l) = (I:f^{l+1})$. Some $l$ with this property always exist, because the ascending chain of ideals $I \subset (I:f) \subset (I:f^2) \subset \dots$ is stationary. We consider $f^0 = 1$, hence $\ell_f(I) = 0$ if and only if $I = (I:f)$. Note that if $l \geq \ell_f(I)$, then $(I:f^l) = (I:f^\infty)$.

Clearly, if $J_i$ are ideals such that $I = \bigcap J_i$, then $\ell_f(I) \leq \sup \ell_f(J_i)$.

\begin{lemma}    \label{lemma:mainlemma}
Let $I$ be an ideal with primary decomposition $I = \bigcap_{i = 1}^k Q_i$, with $f \in P_1, \dots, P_h$ and $f \not\in P_{h+1}, \dots, P_k$. Then
$$\ell_f(I) \leq \sum_{i=0}^h \mult_I(P_i)$$
\end{lemma}
\begin{proof}
Let $L(I) = \sum_{i=0}^h \mult_I(P_i)$. We need to show that $\ell_f(I) \leq L(I)$.

We proceed by induction on $k$, the number of primary components. If $k=1$ then $I=Q_1$ is primary. If $x \not\in P_1$, then by \cite[Lemma 4.4]{AM} $\ell_f(I) = L(I) = 0$. If $x \in P_1$, then the chain
$$Q_1 \subsetneq (Q_1:f) \subsetneq (Q_1:f^2) \subsetneq \dots \subsetneq (Q_1:f^{\ell_f(I)-1}) \subset P_1$$
is a strictly increasing chain of ideals, they are all $P_1$-primary by \cite[Lemma 4.4]{AM}, hence $\ell_f(I) \leq \mult_I(P_1) = L$.

For general $k$, we write $H_i = I_i \cap Q_i$, where $I_i$ is as in subsection \ref{subsez:arit deg}. As $I = \bigcap_{i=1}^k H_i$, then $\ell_f(I) \leq \max_{i=1}^k \ell_f(H_i)$. Moreover for all $i$, $L(H_i) \leq L(I)$. Hence we only need to prove our theorem for ideals of the form $H_i$: ideals having one primary component which is embedded in all the others.

We can suppose that $I  = \bigcap_{i = 1}^k Q_i$, with $P_i \subset P_1$ for all $i > 1$, with $f \in P_1$. Again we write $H_i = I_i \cap Q_i$. Let $s = \max_{i=2}^k \ell_f(H_i)$. Up to reordering we can suppose that $s = \ell_f(H_2)$, and that $f \in P_2$. By inductive hypothesis, as $H_2$ has less primary components than $I$, we know that $s \leq L(H_2)$. If $s = \ell_f(I)$ we are done. We can suppose that $r = \ell_f(I) - s > 0$. Now we will prove that $r \leq \mult_I(P_1)$, and this is enough because then we will have $\ell_f(I) \leq s+r \leq L(H_2) + \mult_I(P_1) \leq L(I)$.

Consider the ideal $(I:f^s) = (Q_1:f^s) \cap \left(\bigcap_{i=2}^k (H_i : f^s)\right)$. For $i \geq 2$, $(H_i:f^s) = (H_i:f^\infty)$. In particular $J = \bigcap_{i=2}^k (H_i : f^s) = \bigcap_{i=h+1}^k Q_i$ by lemma \ref{lemma:iterated division}. Now for $i \geq 0$ we have $(I:f^{s+i})=(Q_1:f^{s+i})\cap J$. The following chain is strictly increasing:
$$(Q_1:f^s)\cap J \subsetneq (Q_1:f^{s+1})\cap J \subsetneq \dots \subsetneq (Q_1:f^{s+r})\cap J $$
By lemma \ref{lemma:strict chains} there is $a \in J$ such that $af^{s+r} \in Q$ and $af^{s+r-1} \not\in Q$. As we said, for every  $i \geq 2$ we have $J \subset (H_i:f^\infty) = (H_i : f^s)$. In particular, for every $i\geq 2$, $af^s \in H_i$. Hence $af^s \in I_1 = \bigcap_{i=2}^k H_i$. If $b=af^s \in I_1$, we have $bf^r \in Q$ and $bf^{r-1} \not\in Q$. By lemma \ref{lemma:strict chains}, we have the strictly ascending chain:
$$Q_1 \cap I_1 \subsetneq (Q_1:f) \cap I_1 \subsetneq \dots \subsetneq (Q_1:f^r) \cap I_1$$
And this implies that $r \leq \mult_I(P_1)$, as required.   
\end{proof}

\begin{lemma}     \label{lemma:monomi contenuti}
Let $I$ be an ideal with primary decomposition $I = \bigcap_{i = 1}^k Q_i$, and let $B = \sum_{i=0}^k \mult_I(P_i)$. Then there exists a monomial $x^\alpha$ of degree not greater than $B$ such that 
$$(I:x^\alpha) = (I : x_0 \cdots x_n^{\infty})$$
In particular, if $I$ contains a monomial, then $I$ contains a monomial of degree not greater than $B$.

If $I$ is saturated, then $B \leq \adeg(I)$. If $I$ is saturated and with Hilbert polynomial $g(m_0, \dots, m_s; x)$, then $B \leq m_0$.
\end{lemma}
\begin{proof}
Up to reordering the primary components, we can suppose that there exist integers $1=h_0 \leq \dots \leq h_{n+1} \leq k$ such that for every $i \in \{0, \dots, n\}$
$$x_i \in P_{h_i}, \dots, P_{h_{i+1}-1} \mbox{ and } x_i \not\in P_{h_{i+1}}, \dots, P_k$$
Now let $\alpha_i = \sum_{i=h_i}^{h_{i+1}-1} \mult_I(P_i)$. By applying repeatedly lemmas \ref{lemma:mainlemma} and \ref{lemma:iterated division} we get  
$$(I:x^\alpha)=(\cdots(I:x_0^{\alpha_0}): \cdots : x_n^{\alpha_n}) = (\cdots(I:x_0^{\infty}): \cdots : x_n^{\infty}) = (I : x_0 \cdots x_n^{\infty})$$
The ideal $I$ contains a monomial if and only if $(I : x_0 \cdots x_n^{\infty}) = (1) = (I:x_0^{\alpha_0} \cdots x_n^{\alpha_n})$ and this happens if and only if $I$ contains $x_0^{\alpha_0} \cdots x_n^{\alpha_n}$. 
\end{proof}

\section{Examples}   \label{sez:examples}

\subsection{Hypersurfaces}
Let $\cappa$ be an algebraically closed field, with a surjective real valued valuation $v:\cappa^* \mapsto \erre$. Fix a projective space $\cappa\pro^n$ and a degree $d$. Consider the Hilbert polynomial $p(x) = g(m_0,\dots,m_{n-1}; x)$ with $m_0 = \dots = m_{n-1} = d$. An ideal $I \subset \cappa[x_0, \dots, x_n]$ has Hilbert polynomial $p$ if and only if $I = (f)$ with $f$ homogeneous with $\deg(f) = d$, hence \hilb{p}{n} is the parameter space of hypersurfaces of $\cappa\pro^n$ of degree $d$. For such ideals $I$, the component $I_d$ of degree $d$ contains only the scalar multiples of $f$, hence it is a one-dimensional linear subspace of $S_d$, and all one-dimensional linear subspaces of $S_d$ are of this form. The Grassmannian of one-dimensional subspaces of $S_d$ is the projective space $\pro(S_d) = \cappa\pro^{N-1}$ with $N:=\binom{n+d-1}{n-1}$. This is the space of projective classes of polynomials of degree $d$. The embedding of the Hilbert scheme described in subsection \ref{subsez:immersion} is just the identification of $\hilb{p}{n}$ with $\cappa\pro^{N-1}$. Its tropicalization is $\tropicalization\left(\hilb{p}{n}\right) =\trop\pro^{N-1}$.

In this case it is possible to understand the correspondence from $\tropicalization\left(\hilb{p}{n}\right)$ to the set of tropical hypersurfaces of degree $d$ quite well. Here we want to underline two facts. One is that this map is not injective. The other is that it is possible to adjust things so that the map becomes injective. It is necessary to add some extra structure to the tropical hypersurfaces, namely to add weights to the maximal faces, as usual. Once this extra structure is considered, there exists a unique subpolyhedron $P \subset \tropicalization\left(\hilb{p}{n}\right)$ such that the restriction of the correspondence to $P$ is bijective.
We think that this property of the existence of a subpolyhedron that is a ``good'' parameter space should probably be true also for the general Hilbert scheme, but in general it is not clear what is the suitable extra structure. In the example of the next subsection we show that in the general case the weights are not enough.

The coordinates in $\cappa\pro^{N-1}$ correspond to the coefficients of the polynomial, hence the tropicalization map $\tau:\cappa\pro^{N-1} \setminus C \mapsto \trop\pro^{N-1}$ sends the projective class of an homogeneous polynomial $f$ in the projective class of the tropical polynomial $\tau(f)$. Note that the tropicalization is defined only on $\cappa\pro^{N-1} \setminus C$, where $C$ is the union of the coordinate hyperplanes. Hence we are dealing only with homogeneous polynomials of degree $d$ containing all the monomials of degree $d$ with a non-zero coefficient. This implies that all these polynomials have the same Newton polytope, a simplex that we will denote by $Q \subset \erre^{n+1}$. Let $A = Q \cap \ze^{n+1}$ be the set of monomials of degree $d$. We use the theory of coherent triangulations (see \cite[Chap. 7, Def. 1.3]{GKZ}) and coherent subdivisions (see \cite[Chap. 7, Def. 2.3]{GKZ}). The secondary fan of $(Q,A)$ (see \cite[Chap. 7, C]{GKZ})) is a subdivision of the space $\trop^N$ with polyhedral cones. Two points in the same projective equivalence class belong to the same cone, hence this subdivision passes to the quotient $\trop\pro^{N-1}$. The maximal cones of the secondary fan correspond to coherent triangulations. The subpolyhedron $P$ is the union of all the maximal cones corresponding to triangulations using all the points of $A$ as vertexes. 

Consider a coherent triangulation not using all the points of $A$ as vertexes, and consider a tropical polynomial $f$ lying in the corresponding cone. The points of $A$ not used in the triangulation correspond to monomials of $f$ that never contribute to the maximum of the polynomial. If you perturb slightly the values of their coefficients, the tropical variety does not change, but the point of $\tropicalization\left(\hilb{p}{n}\right)$ changes. This shows that the correspondence is not injective.  

\subsection{Pairs of points in the tropical projective plane}\label{sez:punti_doppi}
Let $\cappa$ be an algebraically closed field, with a surjective real valued valuation $v:\cappa^* \mapsto \erre$. Consider the projective plane $\cappa\p2$ and the Hilbert polynomial $g(2; x)=2$. A scheme $Z$ has Hilbert polynomial $2$ if and only if $Z$ is a pair of two distinct points or $Z$ has a single point with a tangent space of dimension 1. The ideals of such schemes can be retrieved from their homogeneous component of degree 2, which is a 4-dimensional subspace of the vector space of homogeneous polynomials in 3 variables of degree 2, which has dimension 6. The Grassmannian of 4-dimensional subspaces of $S_2$ is embedded in the projective space $\cappa\p{14}$ and \hilb{2}{2} is isomorphic to the symmetric product of two copies of $\cappa\p2$ blown up along the diagonal. The points outside the exceptional divisor correspond to the pairs of distinct points, while the points on the exceptional divisor correspond to singular schemes. We are interested in the latter.

As we said, the singular schemes we find in this setting have a unique point and a 1-dimensional tangent space at it. These schemes can be reconstructed from the data of the point and a line passing through it. In this example we are interested in considering schemes that have the same support while being different as schemes. We fix a point $p=[a:b:c]\in\cappa\p2$. To consider all the schemes supported at $p$ one has to consider all the lines through $p$, which are given by polynomials of the type $f:=lx_0+mx_1+nx_2$. These lines are parametrized by a $\cappa\p1$ and thus in $\hilb{2}{2}$ the locus of the points parametrizing all the schemes supported at $p$, $\Gamma_p$, contained in the exceptional divisor, is isomorphic to $\cappa\p1$. For simplicity, we will restrict ourselves to the generic case; in particular none of the coordinates of $p$ or the coefficients of $f$ is zero.
We can reduce the number of variables involved: the line has to pass through the point, meaning that $n=-\frac{al+bm}c$, and since $l$,$m$ and $n$ are only defined up to a constant, we can choose $m=1$. A scheme supported at $p$ is then defined by the parameter $l$ and we will denote it by $Z_l$.

Now we want to compute the point of the Hilbert scheme that corresponds to a scheme $Z_l$. This point is determined by the homogeneous component $I_2$ of the ideal $I$ defining $Z_l$. To find a basis for $I_2$ we need four independent polynomials in it and, since $I$ contains $f$, three of these can be $x_0f$, $x_1f$ and $x_2f$. As our fourth generator we choose $(cx_0-ax_2)^2$, which corresponds to one of the lines that we excluded counted twice. Since $f$ cannot be $cx_0-ax_2$ these polynomials are independent and form a basis for $I_2$.
The wedge product of the four polynomials that form the basis is an element of $\bigwedge^4S_2$. It is convenient to fix a basis and work with coordinates.
As a basis of the vector space $S_2$ we choose the monomials in the order $x_0^2$, $x_1^2$, $x_2^2$, $x_1x_2$, $x_0x_1$, $x_0x_1$. As a basis of $\bigwedge^4S_2$ we take the wedge products of four out of the six polynomials above, ordered in the same way, and we order the elements of this base in the lexicographic order.

Once we compute the wedge product and find the coordinates there are only four polynomials with more than one term appearing in these expressions, and they are $al+b$, $al-b$, $al+2b$ and $2al+b$. To express the tropicalization of the coordinates we then need to divide $6$ different cases: the first two where $v(al)$ and $v(b)$ are different and the other four where they are equal and the valuation of one of the four binomials is higher.
For this reason, $\tropicalization(\Gamma_p)$ is the union of $6$ rays coming out of a point $P$. Below we report the coordinates of $P$, in square brackets, and the six integer vectors defining each of the six rays. The letters $A$, $B$ and $C$ stand for the tropicalization of $a$, $b$ and $c$ respectively.
\[
\left[\begin{array}{@{}c@{}}
A+B+C \\ 2B+C \\ B+2C \\
B+2C \\ 3C \\ -A+B+3C \\
3B \\ 2B+C \\ -A+3B+C \\
-A+2B+2C \\ 2A+B \\ 2A+C \\
A+B+C \\ A+2C \\ A+2B
\end{array}\right]
,\ \left(\begin{array}{@{}c@{}}
1 \\ 0 \\ 1 \\
1 \\ 0 \\ 1 \\
1 \\ 1 \\ 1 \\
1 \\ 1 \\ 0 \\
1 \\ 0 \\ 1
\end{array}\right)
,\ \left(\begin{array}{@{}c@{}}
-1 \\ 0 \\ 0 \\ 0 \\ 0 \\ -1 \\ 0 \\ 0 \\ -1 \\ -1 \\ 0 \\ 0 \\ 0 \\ 0 \\ 0
\end{array}\right)
,\ \left(\begin{array}{@{}c@{}}
0 \\ 0 \\ -1 \\ 0 \\ 0 \\ 0 \\ -1 \\ 0 \\ 0 \\ 0 \\ -1 \\ 0 \\ 0 \\ 0 \\ -1
\end{array}\right)
,\ \left(\begin{array}{@{}c@{}}
0 \\ 0 \\ 0 \\ -1 \\ 0 \\ 0 \\ 0 \\ 0 \\ 0 \\ 0 \\ 0 \\ 0 \\ 0 \\ 0 \\ 0
\end{array}\right)
,\ \left(\begin{array}{@{}c@{}}
0 \\ 0 \\ 0 \\ 0 \\ 0 \\ 0 \\ 0 \\ -1 \\ 0 \\ 0 \\ 0 \\ 0 \\ 0 \\ 0 \\ 0
\end{array}\right)
,\ \left(\begin{array}{@{}c@{}}
0 \\ 0 \\ 0 \\ 0 \\ 0 \\ 0 \\ 0 \\ 0 \\ 0 \\ 0 \\ 0 \\ 0 \\ -1 \\ 0 \\ 0
\end{array}\right)
\]

Let us now analyze these six rays and understand which schemes' corresponding points lie on each ray.
The points on the first ray correspond to the schemes $Z_l$ for which $L>B-A$. The tropicalization of the line tangent to $Z_l$ is defined by the tropical polynomial $\tau(f)$ having coefficients $L$, $0$ and $L+A-C$. The same is true for the points on the second ray except that they are those for which $L<B-A$ and the coefficients of $\tau(f)$ are $L$, $0$ and $B-C$. The points on the third ray are those for which $v(al+b)<v(al)=v(b)=B$. For them the coefficients are $B-A$, $0$ and $-v(al+b)$. Note that, together, these three expressions give us all the tropical lines passing through the point $[A:B:C]$, which is $\tropicalization(Z_l)$.

The points on the last three rays have a different behavior: for all of them and for the point $P$ the tropicalization of the line is always the same and it is the one that has its center in $[A:B:C]$.

For any value of $l$ we have that $\tropicalization(Z_l)$ is always the same point $[A:B:C]$. Even if one takes into account the weights (see \cite[sec. 3.4]{MS}), the weight at $[A:B:C]$ is always 2. On the other hand there is a natural way to associate to any of the points of $\tropicalization(\Gamma_p)$ a tropical line through $[A:B:C]$ that is the tropicalization of the tangent line to $Z_l$. This seems to suggest that one should consider some extra structure on $\tropicalization(Z_l)$ that includes the datum of the tropicalization of the tangent line.

\end{document}